\documentclass[11pt]{article}
\usepackage[numbers,sort&compress]{natbib}
\usepackage{enumerate}
\usepackage{amscd}
\usepackage{amsmath}
\usepackage{latexsym}
\usepackage{amsfonts}
\usepackage{setspace}
\usepackage{amssymb}
\usepackage{amsthm}
\usepackage{verbatim}
\usepackage{mathrsfs}
\usepackage{enumerate}
\usepackage[hypertexnames=false]{hyperref}
\usepackage[utf8]{inputenc}
\usepackage{amsmath, amssymb, pgfplots}
\pgfplotsset{compat = newest, width = 12cm, height =12cm}
\usepackage{tikz}
\usepackage{caption}

\theoremstyle{plain}
\theoremstyle{definition}\newtheorem{theorem}{Theorem}[section]
\theoremstyle{definition}\newtheorem{lemma}[theorem]{Lemma}
\theoremstyle{plain}
\theoremstyle{plain}\newtheorem{prop}[theorem]{Proposition}
\theoremstyle{definition}\newtheorem{remark}{Remark}[section]
\usepackage{xcolor}

\allowdisplaybreaks

\numberwithin{equation}{section}
\begin{document}
	\title{Existence of Periodic Solutions to  Steady Viscous Burgers Equation with a General Force}
	\author{Yuhan Cao\footnote{School of Mathematical Sciences, Capital Normal University, Beijing, 100048, P. R. China. Email: caoyuhan@163.com},\ \ \ Quansen Jiu\footnote{School of Mathematical Sciences, Capital Normal University, Beijing, 100048, P. R. China. Email: jiuqs@cnu.edu.cn}}
	\date{}
	\maketitle
	\begin{abstract}
		In this paper, we will construct  periodic solutions to the viscous steady Burgers equation with an external force $f(x)$, based on the  following formal expansion
		\begin{equation*}
			u^\varepsilon (x)=u_0(x)+\varepsilon u_1(x)+\cdots+\varepsilon ^nu_n(x)+\cdots,
		\end{equation*}
		where $\varepsilon\ge 0$ represents the viscosity and $u_0(x)$ is a solution to  non-viscous steady Burgers equation with the external force $f(x)$. We will focus on the solutions  which are uniformly bounded with respect to the viscosity. In \cite{JC}, starting from $u_0=-(2+\cos x)$, the authors   constructed the periodic solutions to the viscous steady Burgers equation with the external force $f=u_0(x)u_{0x}=-2\sin x-\sin x\cos x$. In this paper, we will extend the main result obtained in \cite{JC} and construct the solutions starting from  general $u_0$  and  $f$   satisfying the non-viscous steady Burgers equation  $u_0(x)u_{0x}=f$.  It will be shown that there exists a $\varepsilon_0>0$, which depends on $n$, such that for any $0<\varepsilon<\varepsilon_0$, the viscous steady Burgers equation with the external force $f$ has a periodic solution $u^\varepsilon(x) \in C^2([0,2\pi])$,  satisfying
		$$|u^\varepsilon(x)-u_0(x)-\varepsilon u_1(x)-\cdots-\varepsilon^nu_n(x)| \leq C\varepsilon^{n+1},$$
		where $C>0$ is a constant which depends on $n$, but is independent of $\varepsilon$.

		Compared with Jauslin-Kreiss-Moser’s result in \cite{M}, we present a new approach to construct  periodic solutions to the viscous steady Burgers with a general external force.  The constructed solutions will tend to the ones of the non-viscous Burgers equation with a sharper convergence rate (up to higher order) when the viscosity vanishes.
		
	\end{abstract}
	
	
	\section{Introduction and main results}
	\qquad	The one-dimensional viscous Burgers equation can be written as (see \cite{H})
	\begin{equation}\label{03}
		u_t-\varepsilon u_{xx}+uu_x=f,
	\end{equation}
	where $u=u(x,t)$ is an unknown function, $\varepsilon \geq0$ represents the viscosity coefficient, and $f=f(x,t)$ expresses an external force which is a known function of $x$ and $t$.  When $\varepsilon=0,$ equation \eqref{03} becomes the  non-viscous Burgers equation
	\begin{equation}\label{02}
		u_t+uu_x=f.
	\end{equation}\\
	In \cite{M},  Jauslin, Kreiss, and Moser investigated the existence and uniqueness of the periodic solution to \eqref{03} for some specific smooth forcing function $ f(x,t)$. Among these results, it is also proved in \cite{M} that, when the external force is time-independent,  the equation \eqref{03} admits a  solution satisfying the following steady viscous Burgers equation
	\begin{equation}\label{04}
		\frac{1}{2}(U^2)_x=\varepsilon U_{xx}+F_x,
	\end{equation}
	where $F=F(x)$ is a smooth function of $x$. We state two main results in \cite{M} as follows.
	\begin{prop}{(\cite{M})}\label{prop1}
		Assume that $f=F_x(x,t)$, where $F(x,t)$ is a smooth function of period one in both variables $x$ and $t$. Then equation \eqref{03} with $\varepsilon >0$ has a unique $\mathrm{{Z}^{2}}$-periodic solution satisfying
		$\int_0^1 u ~dx=c$ for any fixed $c \in \mathrm{{R}}$.
	\end{prop}
		
	\begin{prop}{(\cite{M})}\label{prop2}
		Assume that $f=F_x(x)$ depends on $x$ only. Then for any $\varepsilon>0$, the solution presented in Proposition \ref{prop1} converge, for $ t \rightarrow \infty$, to a unique 1-periodic solution $U(x,\varepsilon):=u_{\infty}^{(\varepsilon)}(x)$ of the steady viscous Burgers equation \eqref{04} satisfying
		$$\int_0^1U dx=c.$$
		Moreover, there is a constant $M>0$ which does not depend on $\varepsilon$ such that
		\begin{equation*}
			|U| \leq M , U_x \leq M,  \int_0^1|U_x| dx \leq M.
		\end{equation*}
	\end{prop}
	
	It should be remarked that the periodic solution presented in Proposition \ref{prop2} is uniformly bounded with respect to the viscosity $\varepsilon$. Further studies on the periodic solutions to equation \eqref{03} or \eqref{04} are referred to \cite{E}, \cite{Chen}, \cite{V} and \cite{Z}. We also mention the results in \cite{FTKS}, \cite{Cy} and \cite{CZ} which are concerned with computer-assisted researches.

	In this paper, we consider periodic solutions to the following one-dimensional steady Burgers equation
	\begin{equation}\label{1.1}
		uu_{x}-\varepsilon u_{xx}=f(x), ~x \in [0,2\pi]
	\end{equation}
	where $\varepsilon\ge 0$ is the viscosity coefficient, $u=u(x)$  is the unknown function satisfying
	$$u(0)=u(2\pi).$$
	And $f(x)$ is a given $2\pi$-periodic function  satisfying
	\begin{equation}\label{1.1+}
		\int_{0}^{2\pi} f(x) dx=0.
	\end{equation}
	The aim of this paper is to construct  periodic solutions to the steady Burgers equation\eqref{1.1}, which are uniformly bounded with respect to the viscosity $\varepsilon$. The approach is different from that in \cite{M}.
	Denote the solutions of equation \eqref{1.1} by $u^{\varepsilon}(x)$, which  depend on $\varepsilon$ in general.
	Then we construct the solutions to \eqref{1.1} with the following formal expansion
	\begin{equation}\label{1.2+}
		u^\varepsilon (x)=u_0(x)+\varepsilon u_1(x)+\cdots+\varepsilon ^nu_n(x)+\cdots,
	\end{equation}
where $u_0(x)$ is a solution to the non-viscous steady Burgers equation with the external force $f(x)$, that is, $u_0u_{0x}=f$.
	Substituting \eqref{1.2+} into \eqref{1.1}, we can obtain recurrence relations between the approximate solutions $u_n(x) (n=0,1,\cdots)$ as follows
	\begin{equation}\label{0622-1}
		u_0(x)u_k(x)=u_{k-1}^\prime(x)-\frac{1}{2}\sum_{i+j=k}u_i(x)u_j(x) , ~k=1,2,\cdots.
	\end{equation}
	In \cite{JC}, starting from $u_0=-(2+\cos x)$, the authors   constructed the periodic solutions to the viscous steady Burgers equation with external force $f=u_0(x)u_{0x}=-2\sin x-\sin x\cos x$. Moreover, every $u_n(x) (n=1,2,\cdots)$ is explicitly expressed in \cite{JC}. In this paper, we relax the specific restrictions on $u_0$ in \cite{JC} and construct the solutions starting from the general $u_0$ and  $f$  which satisfy  $u_0(x)u_{0x}=f$.	 In other words, given a periodic function $f(x)$ satisfying $\int_0^{2\pi} f(x) dx=0$,  we first construct  $u_0(x)$  satisfying $u_0(x)<0$ and
	$$u_0u_{0x}=f(x).$$
	Then according to the recurrence relations \eqref{0622-1}, for every positive integer $n$, we can solve out  $u_n(x)$ for $n=1, 2, \cdots$, which are smooth functions with the period $2\pi$. Based on these construction, we will show  that for any positive integer $n \in \mathrm{{N}_{+}}$, the solution to equation \eqref{1.1} admits the following expansion
	$$u^{\varepsilon}(x)=u_0(x)+\varepsilon u_1(x)+\cdots+\varepsilon^n u_n(x)+o(\varepsilon^n).$$
	Our main result can be stated as follows.
	\begin{theorem}\label{th-1.1}
		There exists a $\varepsilon_0>0$, which depends on $n$, such that for any $0<\varepsilon<\varepsilon_0$, the equation \eqref{1.1} has a periodic solution $u^\varepsilon(x) \in C^2([0,2\pi])$,  satisfying
		$$|u^\varepsilon(x)-u_0(x)-\varepsilon u_1(x)-\cdots-\varepsilon^nu_n(x)| \leq C\varepsilon^{n+1}$$
		for all $x\in [0,2\pi]$, where $C>0$ is a constant which depends on $n$, but is independent of $\varepsilon$.
	\end{theorem}
	
	\begin{remark}
		In comparison  with Proposition \ref{prop2}, we present an explicit constructive procedure for the periodic solutions and derive a sharper convergence rate.
	\end{remark}
	
	To prove Theorem \ref{th-1.1},
	we introduce an error function $g(\frac{x}{\varepsilon})$ satisfying
	$$u^{\varepsilon}(x)=u_0(x)+\varepsilon u_1(x)+\cdots+\varepsilon^n u_n(x)+\varepsilon^{n+1}g(\frac{x}{\varepsilon}).$$
	Then it follows that
	\begin{equation}\label{a}
		\begin{aligned}
			&-g^{\prime\prime}(\frac{x}{\varepsilon})+\varepsilon^{n+1} g(\frac{x}{\varepsilon})g^{\prime}(\frac{x}{\varepsilon})+
			\big(u_0(x)+\varepsilon u_1(x)+\varepsilon^2 u_2(x)+\dots+\varepsilon^n u_n(x)\big)g^{\prime}(\frac{x}{\varepsilon})+\\
			&\big(\varepsilon u_0^\prime(x)+\varepsilon^2u_1^\prime(x)+\varepsilon^3 u_2^{\prime}(x)+\dots+\varepsilon^{n+1} u_n^{\prime}(x)\big) g(\frac{x}{\varepsilon})+\\
			&\sum_{k=1}^{n}{\varepsilon^k(u_k(x)u_n^\prime(x)+u_{k+1}(x)u_{n-1}^\prime(x)+\dots+u_n(x)u_k^\prime(x))}-\varepsilon u_n^{\prime\prime}(x)=0.
		\end{aligned}
	\end{equation}
	Let
	$$y=\frac{x}{\varepsilon}.$$
	It deduces that
	\begin{equation}\label{b}
		\begin{aligned}
			&-g^{\prime\prime}(y)+\varepsilon^{n+1} g(y)g^{\prime}(y)
			+\big(u_0(\varepsilon y)+\varepsilon u_1(\varepsilon y)+\dots+\varepsilon^n u_n(\varepsilon y)\big)g^{\prime}(y)\\
			&+\big(\varepsilon u_0^\prime(\varepsilon y)+\varepsilon^2u_1^\prime(\varepsilon y)+\varepsilon^3 u_2^{\prime}(\varepsilon y)+\dots+\varepsilon^{n+1} u_n^{\prime}(\varepsilon y)\big) g(y)
			+F(\varepsilon y)=0.
		\end{aligned}
	\end{equation}
	To prove Theorem \ref{th-1.1}, we only need to prove
	\begin{theorem}\label{th-1.3}
		There exists a $\varepsilon_{0}>0$ such that for any $0< \varepsilon <\varepsilon_{0}$, the equation \eqref{b}
		has a solution $g(y) \in C^2([0,\frac{2\pi}{\varepsilon}])$, satisfying $$g(0)=g(\frac{2\pi}{\varepsilon})$$
		and
		$$|g| \leq C,$$
		where $C>0$ is a constant which depends on $n$, but is independent of $\varepsilon$.
	\end{theorem}	
	To prove Theorem \ref{th-1.3}, we  construct the solution to the initial value problem of \eqref{b} and consider the following Cauchy problem
	\begin{equation}\label{c}
		\left\{
		\begin{array}{l}
			\begin{aligned}
				&-g^{\prime\prime}(y)+\varepsilon^{n+1} g(y)g^{\prime}(y)+\big(u_0(\varepsilon y)+\varepsilon u_1(\varepsilon y)+\dots+\varepsilon^n u_n(\varepsilon y)\big)g^{\prime}(y)\\
				&+\big(\varepsilon u_0^\prime(\varepsilon y)+\varepsilon^2u_1^\prime(\varepsilon y)+\varepsilon^3 u_2^{\prime}(\varepsilon y)+\dots+\varepsilon^{n+1} u_n^{\prime}(\varepsilon y)\big) g(y)
				+F(\varepsilon y)=0,\\
				&g(0)=\alpha_0,\\
				&g^{\prime}(0)=\beta_0,
			\end{aligned}	
		\end{array}
		\right.
	\end{equation}
	where $\alpha_0$ and $\beta_0$ are some given constants. It is easy to obtain that the Cauchy problem \eqref{c} has a unique solution $g(x) \in [0,T^\varepsilon)$, where $T^\varepsilon>0$ (maybe small) is a constant depending on $\varepsilon$ and $n$ .
	
	Then our subsequent proof consists two key ingredients. The first is to prove  the existence interval of solution to \eqref{c} is at least $[0,\frac{2\pi}{\varepsilon}]$ by using upper and lower solutions method. The second is to construct the periodic solution of \eqref{b} satisfying $g(0)= g(\frac{2\pi}{\varepsilon})$ by applying the  implicit function theorem. To this end, we denote  the solution of \eqref{c} by $g(y,\beta_0)$ and define
	$$\Phi(\varepsilon,\beta_0):=-g^{\prime}(\frac{2\pi}{\varepsilon},\beta_0)+\frac{\varepsilon^{n+1}}
	{2}g^2(\frac{2\pi}{\varepsilon},\beta_0)+\beta_0-\frac{\varepsilon^{n+1}}{2}\alpha_0^2.$$
	It is noted that the existence of periodic solution of \eqref{b} satisfying $g(0)= g(\frac{2\pi}{\varepsilon})$ is equivalent to  $\Phi(\varepsilon, \beta_0)=0$. Thus our goal is to apply the implicit function theorem to find  $\beta_0$ which depends  $\varepsilon$ such that $\Phi(\varepsilon, \beta_0)=0.$  The difficulty lies in  proving that $\Phi(\varepsilon,\beta_0)$  and $\frac{\partial \Phi}{\partial \beta_0}$  are continuous around $(0,0)$ and $\frac{\partial \Phi}{\partial \beta_0}(0,0)=1 \neq 0$ such that we are able to apply the implicit function theorem.
	
	The paper is organized as follows. In Section 2, we construct approximate solutions and state our main theorem. In Section 3, we establish the existence of periodic solutions for the first-order error equation and finish the proof of our main result.
	\section{Construction of the approximate solutions} \setcounter{equation}{0}
	\setcounter{theorem}{0}
	In this section, we construct the approximate solutions $u_0(x), u_1(x), \cdots, u_n(x), \cdots$ presented in \eqref{1.2+}.
	Considering the  one-dimensional steady viscous forced Burgers equationn\eqref{1.1} with \eqref{1.1+},  we choose $f(x)\le 0$ and denote the minimum value of $f(x)$  by $$f_0=\min_{x\in [0,2\pi]} f(x)  \leq 0.$$  Then for any $x$,
	$$\int_0^x f(t) dt \geq 2\pi f_0.$$
	Take $C=-4\pi f_0+1>0$ and let
	$$u_0^2(x)=2\int_0^x f(t) dt+C.$$ Then $u_0$ satisfies
	\begin{equation}\label{1.2}
		u_0(x)u_{0x}(x)=f(x).
	\end{equation}
	Letting
	$$u_0(x)=-\sqrt{2\int_0^x f(t) dt+C},$$
	then $u_0$ is a smooth function with the period $2\pi.$
	It is clear that there exist positive constants $0<m_0<M_0$ such that for any $x \in [0,2\pi]$
	$$-M_0 \leq u_0(x) \leq -m_0 <0.$$
	
	To construct the left approximate solutions $u_1(x), u_2(x),\cdots, u_n(x), \cdots$  in \eqref{1.2}, we
	substitute \eqref{1.2} into  \eqref{1.1} to obtain
	\begin{equation}\label{1.4}
		\begin{aligned}
			&\varepsilon[u_0(x)u_{1x}(x)+u_{0x}(x)u_1(x)-u_{0xx}(x)]\\
			+&\varepsilon^2[u_0(x)u_{2x}(x)+u_1(x)u_{1x}(x)+u_2(x)u_{0x}(x)-u_{1xx}(x)]\\
			+&\cdots\\
			+&\varepsilon^k[u_0(x)u_{kx}(x)+u_1(x)u_{(k-1)x}(x)+\cdots+u_{k-1}(x)u_{1x}(x)+u_k(x)u_{0x}(x)-u_{(k-1)xx}(x)]\\
			+&\cdots=0 .
		\end{aligned}
	\end{equation}
	Then according to the different orders  of $\varepsilon$, we have
	\begin{equation}\label{1.5}
		\left\{
		\begin{array}{ll}
			&u_0(x)u_{1x}(x)+u_{0x}(x)u_1(x)-u_{0xx}(x)=0 , \\
			&u_0(x)u_{2x}(x)+u_1(x)u_{1x}(x)+u_2(x)u_{0x}(x)-u_{1xx}(x)=0 , \\
			&\vdots\\
			&u_0(x)u_{kx}(x)+u_1(x)u_{(k-1)x}(x)+\cdots+u_k(x)u_{0x}(x)-u_{(k-1)xx}(x)=0 . \\
			&\vdots\\
		\end{array}
		\right.
	\end{equation}
	Integrating each equation in \eqref{1.5} with respect to $x$  yields
	\begin{equation}\label{1.6}
		\left\{
		\begin{array}{ll}
			&u_0(x)u_1(x)=u_0^\prime(x), \\
			&u_0(x)u_2(x)=u_1^\prime(x)-\frac{1}{2}u^2_1(x), \\
			&\vdots\\
			&u_0(x)u_{2k}(x)=u^\prime_{2k-1}(x)-[u_1(x)u_{2k-1}(x)+\cdots+u_{k-1}(x)u_{k+1}(x)+\frac{1}{2}u^2_k(x)], \\
			&u_0(x)u_{2k+1}(x)=u^\prime_{2k}(x)-[u_1(x)u_{2k}(x)+u_2(x)u_{2k-1}(x)+\cdots+u_k(x)u_{k+1}(x)], \\
			&\vdots\\
		\end{array}
		\right.
	\end{equation}
	where the constant  is taken to be $0$ during integration. It follows from \eqref{1.6} that for any positive integers $i$, $j$, and $k$,
	\begin{equation}\label{1.7}
		u_0(x)u_k(x)=u_{k-1}^\prime(x)-\frac{1}{2}\sum_{i+j=k}u_i(x)u_j(x) .
	\end{equation}
	Since $u_0(x)<0$ is smooth,  each  $u_k$ is a smooth function with the periodic  $2\pi$ and can be solved out  by  \eqref{1.7}.
	
	\section{Proof of main result}
	
	As mentioned in the introduction,  to prove Theorem \ref{th-1.1},
	we introduce an error function $g(\frac{x}{\varepsilon})$ satisfying
	$$u^\varepsilon (x)=u_0(x)+\varepsilon u_1(x)+\varepsilon^2 u_2(x)+\dots+\varepsilon^n u_n(x)+\varepsilon^{n+1}g(\frac{x}{\varepsilon})$$
	satisfies \eqref{1.1}, then the error function $g(\frac{x}{\varepsilon})$ satisfies
	\begin{equation*}
		\begin{aligned}
			-g^{\prime\prime}(\frac{x}{\varepsilon})+\varepsilon^{n+1} g(\frac{x}{\varepsilon})g^{\prime}(\frac{x}{\varepsilon})+
			\big(u_0+\dots+\varepsilon^n u_n\big)(x)g^{\prime}(\frac{x}{\varepsilon})+
			\big(\varepsilon u_0^\prime+\dots+\varepsilon^{n+1} u_n^{\prime}\big)(x) g(\frac{x}{\varepsilon})+\\
			\sum_{k=1}^{n}{\varepsilon^k\big(u_k(x)u_n^\prime(x)+u_{k+1}(x)u_{n-1}^\prime(x)+\dots+u_n(x)u_k^\prime(x)\big)}-\varepsilon u_n^{\prime\prime}(x)=0.
		\end{aligned}
	\end{equation*}
	Denote
	\begin{equation*}
		F(x)=\sum_{k=1}^{n}{\varepsilon^k\big(u_k(x)u_n^\prime(x)+u_{k+1}(x)u_{n-1}^\prime(x)+\dots+u_n(x)u_k^\prime(x)\big)}-\varepsilon u_n^{\prime\prime}(x),
	\end{equation*}
	$$F_1(x)=\int_0^xF(t) dt.$$
	Then $F(x)$ and $F_1(x)$ are smooth and bounded functions with the period $2\pi$.
	
	Let $\displaystyle{y=\frac{x}{\varepsilon}}$. It follows that
	\begin{equation}\label{1.8}
		\begin{aligned}
			-g^{\prime\prime}(y)+\varepsilon^{n+1} g(y)g^{\prime}(y)+\big(u_0\dots+\varepsilon^n u_n\big)(\varepsilon y)g^{\prime}(y)
			+\big(\varepsilon u_0^\prime+\dots+\varepsilon^{n+1} u_n^{\prime}\big)(\varepsilon y) g(y)
			+F(\varepsilon y)=0.
		\end{aligned}
	\end{equation}
	For the convenience of presentation, we set $\varepsilon=\lambda^2$. Then \eqref{1.8} becomes
	\begin{equation}\label{1.9}
		\begin{aligned}
			-g^{\prime\prime}(y)+\lambda^{2n+2} g(y)g^{\prime}(y)+\big(u_0+\dots+\lambda^{2n} u_n\big)(\lambda^2 y)g^{\prime}+\big(\lambda^2 u_0^\prime+\dots+\lambda^{2n+2} u_n^{\prime}\big)(\lambda^2 y) g+F(\lambda^2 y)=0.
		\end{aligned}	
	\end{equation}
	To prove Theorem \ref{th-1.1},	we only need to prove Theorem \ref{th-1.3}, which can be equivalently stated as
	\begin{theorem}\label{2}
		There exists $\lambda_{0}>0$ such that for any $0<\left| \lambda\right| <\lambda_{0}$, the equation \eqref{1.9}
		has a periodic solution $g(y) \in C^2([0,\frac{2\pi}{\lambda^2}])$, satisfying
		$$g(0)=g(\frac{2\pi}{\lambda^2})$$
		and
		$$|g| \leq C,$$
		where $C>0$ is a constant depending on $n$ but independent of $\varepsilon.$
	\end{theorem}
	
	Now we prove Theorem \ref{2}. We first prove the existence of  the initial problem to equation \eqref{1.9}, which is
	\begin{lemma}
		Consider the Cauchy problem
		\begin{equation}\label{1.10}
			\left\{
			\begin{array}{l}
				\begin{aligned}
					&-g^{\prime\prime}(y)+\lambda^{2n+2} g(y)g^{\prime}(y)+\big(u_0+\dots+\lambda^{2n} u_n\big)(\lambda^2 y)g^{\prime}\\&+\big(\lambda^2 u_0^\prime+\dots+\lambda^{2n+2} u_n^{\prime}\big)(\lambda^2 y) g+F(\lambda^2 y)=0,\\
					&g(0)=\alpha_0,\\
					&g^{\prime}(0)=\beta_0.
				\end{aligned}	\\
			\end{array}
			\right.
		\end{equation}
		where $\lambda \neq 0$, $\alpha_0$ and $\beta_0$ are some given initial values. Then there exists a $T^{\varepsilon}>0$ such that  \eqref{1.10} has a unique solution $g(y) \in C^2([0,T^{\varepsilon}))$.
	\end{lemma}
	
	{\bf Proof.} Note that the Cauchy problem \eqref{1.10} is equivalent to the following first-order
	ordinary differential equation
	\begin{equation}\label{1.11}
		\left\{
		\begin{array}{ll}
			\begin{aligned}
				&-g^{\prime}(y)+\frac{\lambda^{2n+2}}{2} g^2(y)+\big(u_0+\dots+\lambda^{2n} u_n\big)(\lambda^2 y)g(y)\\
				&+F_1(\lambda^2 y)+\beta_0-\frac{\lambda^{2n+2}}{2}\alpha_0^2-\big(u_0+\dots+\lambda^{2n} u_n\big)(0)\alpha_0=0,\\
				&g(0)=\alpha_0.\\
			\end{aligned}	\\
		\end{array}
		\right.
	\end{equation}
	Then thanks to the classical Cauchy-Lipschitz theorem, there exists a $T^{\varepsilon}>0$ such that the Cauchy problem \eqref{1.11} admits a unique  solution $g(y)\in C^1[0,T^{\varepsilon})$.$\hfill\Box$
	
	It is remarked that the existence time  $T^{\varepsilon}>0$ in Lemma \ref{1.10} may be small. In the following, we show that under suitable conditions the result in Lemma \ref{1.10} holds for all $t\in [0,\infty)$, which is
	\begin{lemma}\label{La6.2}
		There exists $X_0>0$ and $\lambda_1>0$ such that for every $\alpha_0<-X_0,|\beta_0|<1,0<|\lambda|<\lambda_1$
		the Cauchy problem \eqref{1.11}
		admits a unique solution $g(y)\in C^1([0,+\infty))$. Moreover, the solution $g(y)$
		is  bounded  uniformly with respect to $\lambda$.
	\end{lemma}
	{\bf Proof.} By the boundedness and smoothness of $u_1,u_2\dots u_n$, and $-M_0 \leq u_0(y)\leq -m_0<0$, there exist $0<\lambda_1<1$ and $X_0,M_1>0$, such that for any $y\in[0,\infty)$, if $0<|\lambda|<\lambda_1$ and $\alpha_0<-X_0$, then
	$$-2M_0<u_0(\lambda^2 y)+\lambda^2 u_1(\lambda^2 y)+\lambda^4 u_2(\lambda^2 y)+\dots+\lambda^{2n} u_n(\lambda^2 y)<-\frac{m_0}{2},$$
	and
	$$-M_1 \leq F_1(\lambda^2 y)+\beta_0-\frac{\lambda^{2n+2}}{2}\alpha_0^2-\big(u_0(0)+\lambda^2 u_1(0)+\dots+\lambda^{2n} u_n(0)\big)\alpha_0<0.$$
	
	Let
	$$\displaystyle{ w(y)=e^{-\int_0^y(u_0+\dots+\lambda^{2n} u_n)(\lambda^2 t)dt}g(y)}.$$
	Then $w(y)$ satisfies
	\begin{equation}\label{1.12}
		\left\{
		\begin{array}{l}
			\begin{aligned}
				&w^{\prime}(y)=\frac{\lambda^{2n+2}}{2}e^{\int_0^y(u_0+\dots+\lambda^{2n} u_n)(\lambda^2 t)dt}w^2(y)+\\
				&	e^{-\int_0^y(u_0+\dots+\lambda^{2n} u_n)(\lambda^2 t)dt}
				\big(F_1(\lambda^2 y)+\beta_0-\frac{\lambda^{2n+2}}{2}\alpha_0^2-\big(u_0(0)+\dots+\lambda^{2n} u_n(0)\big)\alpha_0\big),\\
				&	w(0)=g(0)=\alpha_0.
			\end{aligned}	
		\end{array}
		\right.
	\end{equation}
	Then the existence interval of the solution to \eqref{1.11} is same as that of  \eqref{1.12}, of which can be estimated by upper and lower solutions method as follows.
	
	Consider
	\begin{equation}\label{1.13}
		\left\{
		\begin{array}{ll}
			\overline{w}^{\prime}(y)=\overline{w}^2(y),\\
			\overline{w}(0)=\alpha_0.
		\end{array}
		\right.
	\end{equation}
	Solve \eqref{1.13} to get
	$$\displaystyle{\overline{w}(y)=\frac{\alpha_0}{1-\alpha_0y}}$$ for $ y\in[0,\infty).$
	
	Consider
	\begin{equation}\label{1.14}
		\left\{
		\begin{array}{ll}
			\underline{w}^{\prime}(y)=-M_{1}e^{2M_0y},\\
			\underline{w}(0)=\alpha_0.
		\end{array}
		\right.
	\end{equation}
	Solve \eqref{1.14} to obtain
	$$\displaystyle{\underline{w}(y)=\alpha_0+\frac{M_1}{2M_0}(1-e^{2M_0y})}$$ for $ y\in[0,\infty).$
	
	When $y \in[0,+\infty)$, we have
	$$-M_{1}e^{2M_0y} \leq w^{\prime}(y) \leq w^2(y).$$
	Thus it follows that $\underline{w}(y)\leq w(y)\leq \overline{w}(y)$ for any $y\in[0,+\infty)$, which implies that $w(y)\in C^1([0,+\infty))$ and $g(y)\in C^1([0,+\infty))$.
	
	Moreover, for any $0<|\lambda|<\lambda_1$ and $y\in[0,+\infty)$, we claim that
	\begin{equation}\label{07091}	
		-\frac{2M_1}{m_0}\leq g(y)<0.
	\end{equation}
	In fact, for the upper bound, we have
	\begin{equation*}
		\begin{aligned}
			&g(y)=e^{\int_0^y(u_0+\dots+\lambda^{2n} u_n)(\lambda^2 t)dt}w(y)\leq
			e^{\int_0^y(u_0+\dots+\lambda^{2n} u_n)(\lambda^2 t)dt}\overline{w}(y)\\
			&=e^{\int_0^y(u_0+\dots+\lambda^{2n} u_n)(\lambda^2 t)dt}\frac{\alpha_0}{1-\alpha_0y}<0.
		\end{aligned}
	\end{equation*}
	For the lower bound, it follows from \eqref{1.11} that
	\begin{equation*}
		\begin{aligned}
			&g^{\prime}(y)=\frac{\lambda^{2n+2}}{2} g^2(y)+\big(u_0+\dots+\lambda^{2n} u_n\big)(\lambda^2 y)g(y)\\
			&+F_1(\lambda^2 y)+\beta_0-\frac{\lambda^{2n+2}}{2}\alpha_0^2-\big(u_0(0)+\lambda^2 u_1(0)+\dots+\lambda^{2n} u_n(0)\big)\alpha_0>-\frac{m_0}{2}g(y)-M_{1}.
		\end{aligned}
	\end{equation*}
	Then there must be $g^{\prime}>0$, whenever $g<\displaystyle{\frac{-2M_1}{m_0}}$. The claim \eqref{07091} is thus proved and the proof of Lemma \ref{La6.2}  is complete.$\hfill\Box$
	
	Now we are ready to prove Theorem \ref{2}.
	
	{\bf Proof of Theorem \ref{2}.}
	Denote the solution of \eqref{1.11} by $g(y,\beta_0)$, which is presented in Lemma \ref{La6.2}.
	Setting $y=\frac{2\pi}{\lambda^2}$ in \eqref{1.11}. Due to the facts that $u_i (i=1,2,\cdots,n)$ are periodic function, it follows that
	$$	g(\frac{2\pi}{\lambda^2},\beta_0)=g(0,\beta_0)$$
	is equivalent to
	\begin{equation}\label{6.6}		-g^{\prime}(\frac{2\pi}{\lambda^2},\beta_0)+\frac{\lambda^{2n+2}}{2}g^2(\frac{2\pi}{\lambda^2},\beta_0)+\beta_0-\frac{\lambda^{2n+2}}{2}\alpha_0^2=0.
	\end{equation}
	
	Let $$\Phi(\lambda,\beta_0)=-g^{\prime}(\frac{2\pi}{\lambda^{2}},\beta_0)+\frac{\lambda^{2n+2}}{2}g^2(\frac{2\pi}{\lambda^2},\beta_0)+\beta_0-\frac{\lambda^{2n+2}}{2}\alpha_0^2=-(u_0+\dots+\lambda^{2n}u_n)(0)(g(\frac{2\pi}{\lambda^2},\beta_0)-\alpha_0),$$
	where $(\lambda,\beta_0)\in A=\{(\lambda,\beta_0)\big|0< |\lambda|  <\lambda_1, |\beta_0|<1\}$ and $g$ is presented in Lemma \ref{La6.2}.
	
	In the following, we will apply for the implicit theorem to prove 	$\Phi(\lambda,\beta_0)=0$ which is exactly \eqref{6.6}. To this end, the value of $\Phi(0,\beta_0)$  should be fixed. We consider the following initial-value problem:
	\begin{equation}\label{1.16}
		\left\{
		\begin{array}{ll}
			-\tilde{g}^{\prime}(y,\beta_0)+u_0(0)\tilde{g}(y,\beta_0)+\beta_0-u_0(0)\alpha_0=0,\\
			\tilde{g}(0,\beta_0)=\alpha_0.
		\end{array}
		\right.		
	\end{equation}
	Solve \eqref{1.16} to get
	$$\displaystyle{\tilde{g}(y,\beta_0)=\alpha_0-\frac{\beta_0}{u_0(0)}(1-e^{u_0(0)y})}.$$
	Let $r(y,\beta_0)=g(y,\beta_0)-\tilde{g}(y,\beta_0)$. Then it deduces that
	\begin{equation}\label{1.17}
		\left\{
		\begin{array}{ll}
			-r^{\prime}(y,\beta_0)+u_0(0)r(y,\beta_0)+\frac{\lambda^{2n+2}}{2} g^2(y,\beta_0)+\big(u_0(\lambda^2 y)+\dots+\lambda^{2n} u_n(\lambda^2 y)-u_0(0)\big)g(y,\beta_0)\\
			+F_1(\lambda^2 y)-\frac{\lambda^{2n+2}}{2}\alpha_0^2-\big(\lambda^2 u_1(0)+\lambda^4 u_2(0)+\dots+\lambda^{2n} u_n(0)\big)\alpha_0=0,\\
			r(0,\beta_0)=0.
		\end{array}
		\right.		
	\end{equation}
	Solving \eqref{1.17} yields
	\begin{equation*}
		\begin{aligned}
			r(y,\beta_0)=e^{u_0(0)y}\Big(\int_{0}^{y}\big(\frac{\lambda^{2n+2}}{2} g^2(t,\beta_0)+\big(u_0(\lambda^2 t)+\dots+\lambda^{2n} u_n(\lambda^2 t)-u_0(0)\big)g(t,\beta_0)\\
			+F_1(\lambda^2 t)-\frac{\lambda^{2n+2}}{2}\alpha_0^2-\big(\lambda^2 u_1(0)+\lambda^4 u_2(0)+\dots+\lambda^{2n} u_n(0)\big)\alpha_0\big)e^{-u_0(0)t}dt\Big).
		\end{aligned}
	\end{equation*}
	In particular, we have
	\begin{equation}\label{07092}
		\begin{aligned}
			r(\frac{2\pi}{\lambda^2},\beta_0)&=e^{\frac{2\pi u_0(0)} {\lambda^2}}\Big(\int_{0}^{\frac{2\pi}{\lambda^2}}\big(\frac{\lambda^{2n+2}}{2} g^2(t,\beta_0)+\big(u_0(\lambda^2 t)+\dots+\lambda^{2n} u_n(\lambda^2 t)-u_0(0)\big)g(t,\beta_0)\\
			&+F_1(\lambda^2 t)-\frac{\lambda^{2n+2}}{2}\alpha_0^2-\big(\lambda^2 u_1(0)+\lambda^4 u_2(0)+\dots+\lambda^{2n} u_n(0)\big)\alpha_0\big)e^{-u_0(0)t}dt\Big)\\
			&=e^{\frac{2\pi u_0(0)}{\lambda^2}}\Big(\int_{0}^{2\pi}\big(\frac{\lambda^{2n}}{2} g^2(\frac{t}{\lambda^2},\beta_0)+(\frac{u_0(t)-u_0(0)}{\lambda^2}+u_1(t)+\cdots+\lambda^{2n-2} u_n(t))g(\frac{t}{\lambda^2},\beta_0)\\
			&+\frac{F_1(t)}{\lambda^2}-\frac{\lambda^{2n}}{2}\alpha_0^2-\big(u_1(0)+\lambda^2 u_2(0)+\dots+\lambda^{2n-2} u_n(0)\big)\alpha_0\big)e^{\frac{-u_0(0)t}{\lambda^2}}dt\Big).
		\end{aligned}
	\end{equation}
	Since $g(y)$ is uniformly bounded with respect to $\lambda$ and $F_1(t)$ is a smooth function with the periodic  $2\pi$, we obtain
	\begin{equation}\label{07093}
		\begin{aligned}	
			\lim\limits_{\lambda \rightarrow 0}e^{\frac{2\pi u_0(0)}{\lambda^2}}\int_{0}^{2\pi} Ce^{\frac{u_0(0)t}{\lambda^2}} dt =\lim\limits_{\lambda \rightarrow 0}\frac{-C\lambda^2}{u_0(0)}(1-e^{\frac{2\pi u_0(0)}{\lambda^2}})=0,
		\end{aligned}
	\end{equation}				
	and
	\begin{equation}\label{07094}
		\begin{aligned}	
			&\lim\limits_{\lambda \rightarrow 0}e^{\frac{2\pi u_0(0)}{\lambda^2}}\int_{0}^{2\pi} \frac{F_1(t)}{\lambda^2}e^{\frac{-u_0(0)t}{\lambda^2}} dt =\lim\limits_{\lambda \rightarrow 0}-\frac{e^{\frac{2\pi u_0(0)}{\lambda^2}}}{u_0(0)}\int_{0}^{2\pi}F_1(t) d(e^{\frac{-u_0(0)t}{\lambda^2}})\\
			&=\lim\limits_{\lambda \rightarrow 0}\frac{e^{\frac{2\pi u_0(0)}{\lambda^2}}}{u_0(0)}\int_{0}^{2\pi}e^{\frac{-u_0(0)t}{\lambda^2}}F(t) dt=0,\\
		\end{aligned}
	\end{equation}				
	Moreover, we have
	\begin{equation*}
		\begin{aligned}
			\left\vert e^{\frac{2\pi u_0(0)}{\lambda^2}}\int_{0}^{2\pi} \frac{u_0(t)-u_0(0)}{\lambda^2}g(\frac{t}{\lambda^2},\beta_0)e^{\frac{-u_0(0)t}{\lambda^2}} dt \right\vert
			\leq
			e^{\frac{2\pi u_0(0)}{\lambda^2}}\int_{0}^{2\pi} \frac{|u_0(t)-u_0(0)|}{\lambda^2}Ce^{\frac{-u_0(0)t}{\lambda^2}} dt.
		\end{aligned}
	\end{equation*}
	Since $u_0(t)-u_0(0)$ is smooth,  we have that $|u_0(t)-u_0(0)|$ is Lipschitz continuous and thus weakly differentiable. It follows that $$\frac{d}{dt}|u_0(t)-u_0(0)|=\operatorname{sgn}(u_0(t)-u_0(0))u_0^{\prime}(t)$$
	holds in weak sense.
	
	From the definition of weak derivatives, it follows that
	\begin{equation*}
		\begin{aligned}
			e^{\frac{2\pi u_0(0)}{\lambda^2}}\int_{0}^{2\pi} \frac{|u_0(t)-u_0(0)|}{\lambda^2}Ce^{\frac{-u_0(0)t}{\lambda^2}} dt
			&=-\frac{Ce^{\frac{2\pi u_0(0)}{\lambda^2}}}{u_0(0)}\int_{0}^{2\pi} {|u_0(t)-u_0(0)|}d(e^{\frac{-u_0(0)t}{\lambda^2}}) \\
			&=
			\frac{Ce^{\frac{2\pi u_0(0)}{\lambda^2}}}{u_0(0)}\int_{0}^{2\pi}e^{\frac{-u_0(0)t}{\lambda^2}}\cdot \operatorname{sgn}(u_0(t)-u_0(0)) \cdot u^{\prime}_0(t) dt\\
			&\leq
			\frac{Ce^{\frac{2\pi u_0(0)}{\lambda^2}}}{u_0(0)}\int_{0}^{2\pi}e^{\frac{-u_0(0)t}{\lambda^2}} dt.
		\end{aligned}
	\end{equation*}
	Hence
	\begin{equation}\label{07095}
		\begin{aligned}
			\lim\limits_{\l \rightarrow 0}e^{\frac{2\pi u_0(0)}{\lambda^2}}\int_{0}^{2\pi} \frac{u_0(t)-u_0(0)}{\lambda^2}g(\frac{t}{\lambda^2},\beta_0)e^{\frac{-u_0(0)t}{\lambda^2}} dt=0.
		\end{aligned}
	\end{equation}
	It follows from \eqref{07092}-\eqref{07095} that
	$$\lim\limits_{\lambda \rightarrow 0} r(\frac{2\pi}{\lambda^2},\beta_0)=0,$$ which is uniformly with respect to $\beta_0 \in (-1,1)$.
	
	Thus
	$$\lim\limits_{\lambda \rightarrow 0}g(\frac{2\pi}{\lambda^2},\beta_0)=\lim\limits_{\lambda \rightarrow 0}\tilde{g}(\frac{2\pi}{\lambda^2},\beta_0)=\alpha_0-\frac{\beta_0}{u_0(0)}.$$
	So we can define
	
	\begin{equation*}
		\begin{aligned}
			\Phi(0,\beta_0)&=\lim\limits_{\lambda\rightarrow 0}\Phi(\lambda,\beta_0)\\
			&=\lim\limits_{\lambda\rightarrow0}[(u_{0}(0)+\lambda^2u_1(0)+\cdots+\lambda^{2n}u_n(0))(\alpha_0-g(\frac{2\pi}{\lambda^2},\beta_0))]\\
			&=\beta_0,
		\end{aligned}
	\end{equation*}
	for any $\beta_0\in (-1,1).$
	Then $\Phi(\lambda,\beta_0)$ can be defined on  $U=\{(\lambda,\beta_0)\big| |\lambda|  <\lambda_{1}, |\beta_0|<1\}$. Furthermore, we can verify that $\Phi(\lambda,\beta_0)$  is continuous on $U.$ In particular, at the point $(0,0),$
	since $\Phi(\lambda,\beta_0)$ converges to $\beta_0$ as $\lambda \rightarrow 0$ uniformly with respect to $\beta_0$, we have
	\begin{equation}\label{07101}	
		\lim\limits_{\substack{\beta_0\rightarrow 0 \\ \lambda \rightarrow 0}}\Phi(\lambda,\beta_0)=\lim\limits_{\beta_0\rightarrow 0}\lim\limits_{\lambda \rightarrow 0}\Phi(\lambda,\beta_0)=0=\Phi(0,0).
	\end{equation}
	
	It holds that  $\Phi{_{\beta_0}}=\frac{\partial \Phi}{\partial \beta_0}$ which is the derivative with respect to the parameter $\beta_0$  is continuous on $A=\{(\lambda,\beta_0)\big|0< |\lambda|  <\lambda_1, |\beta_0|<1\}$. Actually, we can prove that $\Phi{_{\beta_0}}$  is well-defined on $U$. To this end, we will first  fix the value of $\frac{\partial \Phi}{\partial \beta_0}$ at $\lambda =0$. Differentiate with respect to $\beta_0$ on both sides of equation \eqref{1.11} to obtain
	\begin{equation}\label{1.18}
		\left\{
		\begin{array}{ll}
			-\frac{d}{d y}\frac{\partial g(y,\beta_0)}{\partial \beta_0}+{\lambda^{2n}} g(y,\beta_0)\frac{\partial g(y,\beta_0)}{\partial \beta_0}+(u_0(\lambda^2 y)+\cdots+\lambda^{2n} u_n(\lambda^2 y))\frac{\partial g(y,\beta_0)}{\partial \beta_0}+1=0,\\
			\frac{\partial g}{\partial \beta_0}(0,\beta_0)=0.\\
		\end{array}
		\right.
	\end{equation}
	
	Note that
	$$-\frac{2M_1}{m_0} \leq g(y,\beta_0)<0, \qquad -2M_0< u_0(\lambda^2y)+\lambda^2u_1(\lambda^2y)+\cdots+\lambda^{2n}u_n(\lambda^2 y)<-\frac{m_0}{2}$$
	for $y \in [0, +\infty)$ and $\beta_0 \in (-1,1).$
	
	Consider the following initial-value problem:
	\begin{equation}\label{1.19}
		\left\{
		\begin{array}{ll}
			\frac{d}{d y}\frac{\partial \overline{g}(y,\beta_0)}{\partial \beta_0}=-\frac{m_0}{2}\frac{\partial \overline{g}(y,\beta_0)}{\partial \beta_0}+1,\\
			\frac{\partial \overline{g}}{\partial \beta_0}(0,\beta_0)=0.\\
		\end{array}
		\right.
	\end{equation}
	Solve \eqref{1.19} to get
	$$\frac{\partial \overline{g}}{\partial \beta_0}(y,\beta_0)=-\frac{2}{m_0}(e^{-\frac{m_0}{2}y}-1)$$
	for $\beta_0 \in (-1,1)$ and $y \in [0,\infty).$
	
	Consider the following initial-value problem:
	\begin{equation}\label{6.13}
		\left\{
		\begin{array}{ll}
			\frac{d}{d y}\frac{\partial \underline{g}(y,\beta_0)}{\partial \beta_0}=-(2M_0+\frac{2M_1}{m_0})\frac{\partial \underline{g}(y,\beta_0)}{\partial \beta_0}+1,\\
			\frac{\partial \underline{g}}{\partial \beta_0}(0,\beta_0)=0.
		\end{array}
		\right.
	\end{equation}
	Solve \eqref{6.13} to get
	$$\frac{\partial \underline{g}}{\partial \beta_0}(y,\beta_0)=\frac{-1}{2M_0+\frac{2M_1}{m_0}}(e^{-(2M_0+\frac{2M_1}{m_0})y}-1)$$
	for $\beta_0 \in (-1,1)$ and $y \in [0,\infty).$
	
	Applying the comparison theorem, we obtain
	$$ 0 \leq \frac{\partial \underline{g}}{\partial \beta_0}(y,\beta_0) \leq \frac{\partial g}{\partial \beta_0}(y,\beta_0) \leq \frac{\partial \overline{g}}{\partial \beta_0}(y,\beta_0)\leq \frac{2}{m_0},$$
	for all $y \in [0, +\infty)$ and $\beta_0 \in (-1,1)$,  which means that the solution $\frac{\partial g}{\partial \beta_0}(y,\beta_0)$ of equation \eqref{1.18} is uniformly bounded on $(y,\beta)\in A$.
	
	Since $g$ is continuously differentiable with respect to  $(\lambda, \beta_0)\in A$, we have
	\begin{align*}
		\frac{\partial r}{\partial  \beta_0}(\frac{2\pi}{\lambda^2},\beta_0)
		&=e^{\frac{2\pi u_0(0)}{\lambda^2}}\Big(\int_{0}^{2\pi}\big({\lambda^{2n}} g(\frac{t}{\lambda^2},\beta_0)\frac{\partial g}{\partial \beta_0}(\frac{t}{\lambda^2},\beta_0)\\
		&+(\frac{u_0(t)-u_0(0)}{\lambda^2}+u_1(t)+\cdots+\lambda^{2n-2}u_n(t))\frac{\partial g}{\partial \beta_0}(\frac{t}{\lambda^2},\beta_0)
		\big)e^{\frac{-u_0(0)t}{\lambda^2}}dt\Big).
	\end{align*}
	Due to the uniform boundedness of $\frac{\partial g}{\partial \beta_0}(\frac{2\pi}{\lambda^2},\beta_0)$, we apply a similar process as in \eqref{07092}-\eqref{07095} to obtain
	$$\lim\limits_{\lambda \rightarrow 0}\frac{\partial r}{\partial \beta_0}(\frac{2\pi} {\lambda^2},\beta_0)=0,$$
	which holds uniformly on $\beta_0 \in (-1,1).$
	
	Thus we can define
	\begin{align*}
		\frac{\partial \Phi}{\partial \beta_0}(0,\beta_0)=\lim\limits_{\lambda \rightarrow 0} \frac{\partial \Phi}{\partial \beta_0}(\lambda^2,\beta_0)
		&=\lim\limits_{\lambda \rightarrow 0}-(u_0+\cdots+\lambda^{2n}u_n)(0)\cdot \frac{\partial g}{\partial \beta_0}(\frac{2\pi}{\lambda^2},\beta_0)\\
		&=\lim\limits_{\lambda \rightarrow 0}-(u_0+\cdots+\lambda^{2n}u_n)(0)\cdot \frac{\partial \tilde{g}}{\partial \beta_0}(\frac{2\pi}{\lambda^2},\beta_0)\\
		&=\lim\limits_{\lambda \rightarrow 0}\frac{(u_0+\cdots+\lambda^{2n}u_n)(0)}{u_0(0)}\cdot(1-e^{\frac{2\pi u_0(0)}{\lambda^2}})\\
		&=1
	\end{align*}
	for any $\beta_0 \in (-1,1).$
	
	Furthermore, similar to \eqref{07101},we can prove that $\frac{\partial \Phi(\lambda,\beta_0)}{\partial \beta_0}$ is continuous on $U$. Since $$\Phi(0,0)=0,  \frac{\Phi(0,0)}{\partial \beta_0}=1 \neq 0,$$
	by the implicit function theorem, there exists $0<\lambda_0<\lambda_1$ such that for every $0<|\lambda|<\lambda_0$, there exists a unique $\beta_0$ satisfying $\Phi(\lambda,\beta_0(\lambda))=0$, that is (\ref{6.6}) holds.
	
	The proof of Theorem  \ref{2} is complete.$\hfill\Box$

	{\bf Acknowledgements.}
	This research was  partially supported by the National Natural Science Foundation of China (NNSFC) (No. 12671278) and the Interdisciplinary Project of Capital Normal University(No. 2026JCYY04).

\end{document}